\documentclass[11pt,reqno]{amsart}

\usepackage{amsmath, amsthm, amstext, amssymb, amsfonts, graphicx, bm, bbm, subfigure, fancyhdr}
\newtheorem{theorem}{Theorem}[section]

\newtheorem*{remark}{Remark}

\numberwithin{equation}{section}

\newcommand{\PR}{\mathbb P}
\newcommand{\ER}{\mathbb E}
\newcommand{\VR}{\mathbb V \text{ar}}
\newcommand{\D}{\mathcal D}
\newcommand{\G}{\mathcal G}

\begin{document}

\title[Exact asymptotics for constrained exponential random graphs]{Large deviations and exact asymptotics for constrained exponential
random graphs}

\author{Mei Yin}

\thanks{Mei Yin's research was partially supported by NSF grant DMS-1308333.}

\address{Department of Mathematics, University of Denver, Denver, CO 80208, USA}

\email{mei.yin@du.edu}

\dedicatory{\rm \today}



\begin{abstract}
We present a technique for approximating generic normalization constants subject to constraints. The method is then applied to derive the exact asymptotics
for the conditional normalization constant of constrained exponential random graphs.
\end{abstract}

\maketitle

\section{Introduction}
\label{intro}
Exponential random graph models are widely used to characterize the structure and behavior of real-world networks as they are able to predict the global structure of the networked system based on a set of tractable local features. Let $s$ be a positive integer. We recall the definition of an $s$-parameter family of exponential random graphs. Let $H_1,\dots,H_s$ be fixed finite simple graphs (``simple'' means
undirected, with no loops or multiple edges). By convention, we take $H_1$ to be a single edge. Let $\zeta_1,\dots,\zeta_s$ be $s$ real parameters and let $N$ be a positive integer. Consider the set $\G_N$ of all simple graphs
$G_N$ on $N$ vertices. Let $\text{hom}(H_i, G_N)$ denote the number of homomorphisms (edge-preserving vertex maps) from the vertex set $V(H_i)$ into the vertex set $V(G_N)$ and $t(H_i, G_N)$ denote the homomorphism density of $H_i$ in $G_N$,
\begin{equation}
\label{t} t(H_i, G_N)=\frac{|\text{hom}(H_i,
G_N)|}{|V(G_N)|^{|V(H_i)|}}.
\end{equation}
By an $s$-parameter family of exponential
random graphs we mean a family of probability measures
$\PR_N^{\zeta}$ on $\G_N$ defined by, for $G_N\in\G_N$,
\begin{equation}
\label{pmf} \PR_N^{\zeta}(G_N)=\exp\left(N^2\left(\zeta_1
t(H_1,G_N)+\cdots+
  \zeta_s t(H_s,G_N)-\psi_N^{\zeta}\right)\right),
\end{equation}
where the parameters $\zeta_1,\dots,\zeta_s$ are used to tune the densities of different subgraphs $H_1,\dots,H_s$ of $G_N$ and $\psi_N^{\zeta}$ is the normalization constant,
\begin{equation}
\label{psi} \psi_N^{\zeta}=\frac{1}{N^2}\log\sum_{G_N \in
\G_N} \exp\left(N^2 \left(\zeta_1
t(H_1,G_N)+\cdots+\zeta_s t(H_s,G_N)\right) \right).
\end{equation}

These exponential models are analogues of grand canonical ensembles in statistical physics, with particle and energy densities in place of subgraph densities, and temperature and chemical potentials in place of tuning parameters. A key objective while studying these models is to evaluate the normalization constant. It encodes essential information about the model since averages of various quantities of interest may be obtained by differentiating the normalization constant with respect to appropriate parameters. Indeed, a phase is commonly characterized as a connected region of the parameter space, maximal for the condition that the limiting normalization constant is analytic, and phase boundaries are determined by examining the singularities of its derivatives. Computation of the normalization constant is also important in statistics because it is crucial for carrying out maximum likelihood estimates and Bayesian inference of unknown parameters. The computation though is not always reliable for large $N$. For example, as shown by Chatterjee and Diaconis \cite{CD}, when $s=2$ and $\zeta_2>0$, all graphs drawn from the exponential model (\ref{pmf}) are not appreciably different from Erd\H{o}s-R\'{e}nyi in the large $N$ limit.

This implies that sometimes subgraph densities cannot be tuned in the unconstrained model and exponential random graphs alone may not capture all desirable features of the networked system, such as interdependency and clustering. Furthermore, unlike standard statistical physics models, the equivalence of various ensembles (microcanonical, canonical, grand canonical) in the asymptotic regime does not hold in these models. One possible explanation is that since the normalization constant in the microcanonical ensemble is not always a convex function of the parameters \cite{RS}, the Legendre transform between the normalization constants in different ensembles is not invertible (see \cite{TET} for discussions about non-equivalence of ensembles). We are thus motivated to study the constrained exponential random graph model in \cite{KY}, where some subgraph density is controlled directly and others are tuned with parameters. In contrast to the above example where in the limit as $N\rightarrow \infty$, all graphs are close to Erd\H{o}s-R\'{e}nyi as $\zeta_2$ increases from $0$ to $\infty$, it was shown in \cite{KY} that for fixed edge density, a typical graph drawn from the constrained edge-triangle model still exhibits Erd\H{o}s-R\'{e}nyi structure for $\zeta_2$ close to $0$, but consists of one big clique and some isolated vertices as $\zeta_2$ gets sufficiently close to infinity. Notice that the transition observed in the constrained model is between graphs of different characters, whereas in the unconstrained model, although there is a curve in the parameter space across which the graph densities display sudden jumps \cite{CD, RY}, the transition is between graphs of similar characters (Erd\H{o}s-R\'{e}nyi graphs). Interesting mathematics is therefore expected from studying the constrained model, and in particular, the associated normalization constant directly; the normalization constant in the unconstrained model may sometimes be of no particular relevance.

For clarity, we assume that the edge density of the graph is approximately known to be $e$, though
the argument runs through without much modification if the density of some other more complicated subgraph is approximately described. Take
$t>0$. The conditional normalization constant
$\psi^{e, \zeta}_{N,t}$ is defined analogously to the normalization
constant $\psi^{\zeta}_{N}$ for the unconstrained exponential random graph model,
\begin{equation}
\label{cpsi1} \psi^{e, \zeta}_{N,t}=\frac{1}{N^2}\log\sum_{G_N\in
\mathcal{G}_N: |e(G_N)-e|\leq t}\exp\left(N^2 \left(\zeta_1
t(H_1,G_N)+\cdots+\zeta_s t(H_s,G_N)\right)\right),
\end{equation}
the difference being that we are only taking into account graphs
$G_N$ whose edge density $e(G_N)$ is within a $t$
neighborhood of $e$. Correspondingly, the associated conditional
probability measure $\PR^{e, \zeta}_{N,t}(G_N)$ is given by
\begin{equation}
\label{cpmf} \PR^{e, \zeta}_{N,t}(G_N)=\exp\left(N^2 \left(\zeta_1
t(H_1,G_N)+\cdots+\zeta_s t(H_s,G_N)-\psi^{e, \zeta}_{N,t}\right)\right)\mathbbm{1}_{|e(G_N)-e| \leq t}.
\end{equation}

Based on a large deviation principle for Erd\H{o}s-R\'{e}nyi graphs established in Chatterjee and Varadhan \cite{CV}, Chatterjee and Diaconis \cite{CD} developed an asymptotic approximation for the normalization constant $\psi_N^{\zeta}$ as $N\rightarrow \infty$ and connected the occurrence of a phase transition in the dense exponential model with the non-analyticity of the asymptotic limit of $\psi_N^{\zeta}$. Further investigations quickly followed, see for example \cite{AR, RRS, RS1, RS, RY, YRF, Z}. However, since the approximation relies on Szemer\'{e}di's regularity lemma, the error bound on $\psi_N^{\zeta}$ is of the order of some negative power of
\begin{eqnarray}
\log^* N=\left\{%
\begin{array}{ll}
    0, & \hbox{if $N\leq 1$;} \\
    1+\log^*(\log N), & \hbox{if $N>1$,} \\
\end{array}%
\right.
\end{eqnarray}
which is the number of times the logarithm function must be iteratively applied before the result is less than or equal to $1$, and this method is also not applicable for sparse exponential random graphs. Analogously, using the large deviation principle established in Chatterjee and
Varadhan \cite{CV} and Chatterjee and Diaconis \cite{CD}, we developed an asymptotic approximation for the conditional normalization constant $\psi^{e, \zeta}_{N,t}$ as $N \rightarrow \infty$ and $t \rightarrow 0$, since it is in this limit that interesting singular behavior occurs \cite{KY}. Nevertheless, this approximation suffers from the same problem: the error bound on $\psi^{e, \zeta}_{N,t}$ is of the order of some negative power of $\log^* N$ and the method does not lead to an exact limit for $\psi^{e, \zeta}_{N,t}$ in the sparse setting.

To improve on the approximation, Chatterjee and Dembo \cite{CD1} presented a general technique for computing large deviations of nonlinear functions of independent Bernoulli random variables in a recent work. In detail, let $f$ be a function from $[0,1]^n$ to $\mathbb{R}$, they considered a generic normalization constant of the form
\begin{equation}
\label{free}
F=\log \sum_{x\in \{0,1\}^n}e^{f(x)}
\end{equation}
and investigated conditions on $f$ such that the approximation
\begin{equation}
\label{valid}
F=\sup_{x\in [0,1]^n}(f(x)-I(x))+\text{ lower order terms}
\end{equation}
is valid, where $I(x)=\sum_{i=1}^n I(x_i)$ and
\begin{equation}
\label{I}
I(x_i)=\sum_{i=1}^n (x_i\log x_i+(1-x_i)\log(1-x_i)).
\end{equation}
They then applied the general result and obtained bounds for the normalization constant $\psi_N^{\zeta}$ for finite $N$, which leads to a variational formula for the asymptotic normalization of exponential random graphs with a small amount of sparsity. Serious attempts have also been made at formulating a suitable ``sparse'' version of Szemer\'{e}di's lemma \cite{BCCZ1,BCCZ2}. This however may not always provide the precision required for large deviations, since random graphs do not necessarily satisfy the proposed regularity conditions in the large deviations regime. Seeing the power of nonlinear large deviations in deriving a concrete error bound for $\psi^{\zeta}_N$ as $N\rightarrow \infty$, we naturally wonder if it is possible to likewise obtain a better estimate for $\psi^{e, \zeta}_{N,t}$ as $N \rightarrow \infty$ and $t \rightarrow 0$, which will shed light on constrained exponential random graphs with sparsity. The following sections will be dedicated towards this goal. Due to the imposed constraint, instead of working with a generic normalization constant of the form (\ref{free}) as in Chatterjee and Dembo \cite{CD1}, we will work with a generic conditional normalization constant in Theorem \ref{main1} and then apply this result to derive a concrete error bound for the conditional normalization constant $\psi^{e, \zeta}_{N,t}$ of constrained exponential random graphs in Theorems \ref{general} and \ref{special}.

\section{Overview of Chatterjee-Dembo results}
\label{overview}
Chatterjee and Dembo came up with a two-part sufficient condition under which the approximation (\ref{valid}) holds. They first assumed that $f$ is a twice continuously differentiable function on $[0,1]^n$ and introduced some shorthand notation. Let $\Vert \cdot \Vert$ denote the supremum norm. For each $i$ and $j$, let
\begin{equation}
f_i=\frac{\partial f}{\partial x_i} \text{ and } f_{ij}=\frac{\partial^2 f}{\partial x_i \partial x_j}
\end{equation}
and define $a=\Vert f \Vert$, $b_i=\Vert f_i \Vert$, and $c_{ij}=\Vert f_{ij} \Vert$. In addition to this minor smoothness condition on the function $f$, they further assumed that the gradient vector $\nabla f(x)=(\partial f/\partial x_1, \dots,\partial f/\partial x_n)$ satisfies a low complexity gradient condition: For any $\epsilon>0$, there is a finite subset of $\mathbb{R}^n$ denoted by $\D(\epsilon)$ such that for all $x\in [0,1]^n$, there exists $d=(d_1,\dots,d_n) \in \D(\epsilon)$ with
\begin{equation}
\label{D}
\sum_{i=1}^n (f_i(x)-d_i)^2 \leq n\epsilon^2.
\end{equation}

\begin{theorem} [Theorem 1.5 in \cite{CD1}]
\label{CD1}
Let $F$, $a$, $b_i$, $c_{ij}$, and $\D(\epsilon)$ be defined as above. Let $I$ be defined as in (\ref{I}). Then for any $\epsilon>0$, $F$ satisfies the upper bound
\begin{equation}
F \leq \sup_{x \in [0,1]^n} (f(x)-I(x)) +\text{ complexity term }+\text{ smoothness term},
\end{equation}
where
\begin{equation}
\text{complexity term }=\frac{1}{4}\left(n\sum_{i=1}^n b_i^2\right)^{1/2}\epsilon+3n\epsilon+\log|\D(\epsilon)|, \text{ and}
\end{equation}
\begin{equation}
\text{smoothness term }=4\left(\sum_{i=1}^n (ac_{ii}+b_i^2)+\frac{1}{4}\sum_{i,j=1}^n (ac_{ij}^2+b_ib_jc_{ij}+4b_ic_{ij})\right)^{1/2}
\end{equation}
\begin{equation*}
+\frac{1}{4}\left(\sum_{i=1}^n b_i^2\right)^{1/2}\left(\sum_{i=1}^n c_{ii}^2\right)^{1/2}+3\sum_{i=1}^n c_{ii}+\log 2.
\end{equation*}

\vskip.2truein

\noindent Moreover, $F$ satisfies the lower bound
\begin{equation}
F\geq \sup_{x\in [0,1]^n}(f(x)-I(x))-\frac{1}{2}\sum_{i=1}^n c_{ii}.
\end{equation}
\end{theorem}

To utilize Theorem \ref{CD1} in the exponential random graph setting, Chatterjee and Dembo introduced an equivalent definition of the homomorphism density so that the normalization constant for exponential random graphs (\ref{psi}) takes the same form as the generic normalization constant (\ref{free}). This notion of the homomorphism density, which dates back to Lov\'{a}sz, is denoted by $t(H,x)$ and may be constructed not only for simple graphs but also for more general objects (referred to as ``graphons'' in Lov\'{a}sz \cite{Lov}). Let $k$ be a positive integer and let $H$ be a finite simple graph on the vertex set $[k]=\{1,\dots,k\}$. Let $E$ be the set of edges of $H$ and let $m=|E|$. Let $N$ be another positive integer and let $n=\binom N2$. Index the elements of $[0,1]^n$ as $x=(x_{ij})_{1\leq i<j\leq N}$ with the understanding that if $i<j$, then $x_{ji}$ is the same as $x_{ij}$, and for all $i$, $x_{ii}=0$. Let $t(H,x)=T(x)/N^2$, where $T: [0,1]^n \rightarrow \mathbb{R}$ is defined as
\begin{equation}
\label{T}
T(x)=\frac{1}{N^{k-2}}\sum_{q\in [N]^k}\prod_{\{l,l'\}\in E}x_{q_l q_{l'}}.
\end{equation}
For any graph $G_N$, if $x_{ij}=1$ means there is an edge between the vertices $i$ and $j$ and $x_{ij}=0$ means there is no edge, then $t(H,x)=t(H,G_N)$, where $t(H,G_N)$ is the homomorphism density defined by (\ref{t}). Furthermore, if we let $G_x$ denote the simple graph whose edges are independent, and edge $(i,j)$ is present with probability $x_{ij}$ and absent with probability $1-x_{ij}$, then $t(H,x)$ gives the expected value of $t(H, G_x)$. Chatterjee and Dembo checked that $T(x)$ satisfies both the smoothness condition and the low complexity gradient condition as assumed in Theorem \ref{CD1}.  In detail, they showed in Lemmas 5.1 and 5.2 of \cite{CD1} that
\begin{equation}
\label{T1}
\Vert T \Vert \leq N^2, \hspace{0.5cm} \Vert \frac{\partial T}{\partial x_{ij}} \Vert \leq 2m,
\end{equation}
\begin{eqnarray}
\label{T2}
\left \Vert \frac{\partial^2 T}{\partial x_{ij}\partial x_{i'j'}} \right \Vert \leq \left\{%
\begin{array}{ll}
    4m(m-1)N^{-1}, & \hbox{if $|\{i,j,i',j'\}|=2$ or $3$;} \\
    4m(m-1)N^{-2}, & \hbox{if $|\{i,j,i',j'\}|=4$,} \\
\end{array}%
\right.
\end{eqnarray}
and for any $\epsilon>0$,
\begin{equation}
\label{T3}
|\D_T(\epsilon)| \leq \exp\left( \frac{cm^4k^4N}{\epsilon^4}\log \frac{Cm^4k^4}{\epsilon^4}\right),
\end{equation}
where $c$ and $C$ are universal constants. By taking $f(x)=\zeta_1T_1(x)+\cdots+\zeta_sT_s(x)$ in Theorem \ref{CD1}, they then gave a concrete error bound for the normalization constant $\psi_N^{\zeta}$, which is seen to be $F/N^2$ in this alternative interpretation of (\ref{free}). This error bound is significantly better than the negative power of $\log^* N$ and allows a small degree of
sparsity for $\zeta_i$. As Theorem \ref{previous} shows, the difference between $\psi_N^{\zeta}$ and the approximation $\sup_{x\in [0,1]^n}\frac{f(x)-I(x)}{N^2}$ tends to zero as long as $\sum_{i=1}^s |\zeta_i|$ grows slower than $N^{1/8}(\log N)^{-1/8}$.

\begin{theorem} [Theorem 1.6 in \cite{CD1}]
\label{previous}
Let $s$ be a positive integer and $H_1,\dots,H_s$ be fixed finite simple graphs. Let $N$ be another positive integer and let $n=\binom N2$. Define $T_1,\dots,T_s$ accordingly as in the above paragraph. Let $\zeta_1,\dots,\zeta_s$ be $s$ real parameters and define $\psi_N^{\zeta}$ as in (\ref{psi}). Let $f(x)=\zeta_1 T_1(x)+\cdots+\zeta_s T_s(x)$, $B=1+|\zeta_1|+\cdots+|\zeta_s|$, and $I$ be defined as in (\ref{I}). Then
\begin{equation}
-cBN^{-1}\leq \psi_N^{\zeta}-\sup_{x\in [0,1]^n}\frac{f(x)-I(x)}{N^2}
\end{equation}
\begin{equation*}
\leq CB^{8/5}N^{-1/5}(\log N)^{1/5}\left(1+\frac{\log B}{\log N}\right)+CB^2N^{-1/2},
\end{equation*}
where $c$ and $C$ are constants that may depend only on $H_1,\dots,H_s$.
\end{theorem}

\section{Nonlinear large deviations}
\label{ld}
Let $f$ and $h$ be two continuously differentiable functions from $[0,1]^n$ to $\mathbb{R}$. Assume that $f$ and $h$ satisfy both the smoothness condition and the low complexity gradient condition described at the beginning of this paper. Let $a$, $b_i$, $c_{ij}$ be the supremum norms of $f$ and let $\alpha$, $\beta_i$, $\gamma_{ij}$ be the corresponding supremum norms of $h$. For any $\epsilon>0$, let $\D_f(\epsilon)$ and $\D_h(\epsilon)$ be finite subsets of $\mathbb{R}^n$ associated with the gradient vectors of $f$ and $h$ respectively. Take $t>0$. Consider a generic conditional normalization constant of the form
\begin{equation}
\label{F}
F^c=\log \sum_{x\in \{0,1\}^n: |h(x)|\leq tn}e^{f(x)}.
\end{equation}

\begin{theorem}
\label{main1}
Let $F^c$,  $a$, $b_i$, $c_{ij}$, $\alpha$, $\beta_i$, $\gamma_{ij}$, $\D_f(\epsilon)$, and $\D_h(\epsilon)$ be defined as above. Let $I$ be defined as in (\ref{I}). Let $K=\log 2+2a/n$. Then for any $\delta>0$ and $\epsilon>0$, $F^c$ satisfies the upper bound

\vspace{0.05cm}

\begin{equation}
F^c\leq \sup_{x\in [0,1]^n: |h(x)|\leq (t+\delta)n}(f(x)-I(x))+\text{ complexity term }+\text{ smoothness term},
\end{equation}
where
\begin{equation}
\text{complexity term }=\frac{1}{4}\left(n\sum_{i=1}^n m_i^2\right)^{1/2}\epsilon+3n\epsilon+\log\left(\frac{12K\left(\frac{1}{n}\sum_{i=1}^n \beta_i^2\right)^{1/2}}{\delta\epsilon}\right)
\end{equation}
\begin{equation*}
+\log|\D_f(\epsilon/3)|+
\log|\D_h((\delta\epsilon)/(6K))|, \text{ and}
\end{equation*}
\begin{equation}
\text{smoothness term }=4\left(\sum_{i=1}^n (ln_{ii}+m_i^2)+\frac{1}{4}\sum_{i,j=1}^n (ln_{ij}^2+m_im_jn_{ij}+4m_in_{ij})\right)^{1/2}
\end{equation}
\begin{equation*}
+\frac{1}{4}\left(\sum_{i=1}^n m_i^2\right)^{1/2}\left(\sum_{i=1}^n n_{ii}^2\right)^{1/2}+3\sum_{i=1}^n n_{ii}+\log 2,
\end{equation*}
where
\begin{equation}
l=a+nK,
\end{equation}
\begin{equation}
m_i=b_i+\frac{2K \beta_i}{\delta},
\end{equation}
\begin{equation}
n_{ij}=c_{ij}+\frac{2K \gamma_{ij}}{\delta}+\frac{6K \beta_i \beta_j}{n\delta^2}.
\end{equation}

\vskip.2truein

\noindent Moreover, $F^c$ satisfies the lower bound
\begin{equation}
F^c\geq \sup_{x\in [0,1]^n: |h(x)|\leq (t-\delta_0)n}(f(x)-I(x))-\epsilon_0 n-\eta_0 n-\log 2,
\end{equation}
where
\begin{equation}
\delta_0=\frac{\sqrt{6}}{n}\left(\sum_{i=1}^n (\alpha \gamma_{ii}+\beta_i^2)\right)^{1/2},
\end{equation}
\begin{equation}
\epsilon_0=2\sqrt{\frac{6}{n}},
\end{equation}
\begin{equation}
\eta_0=\frac{\sqrt{6}}{n}\left(\sum_{i=1}^n (a c_{ii}+b_i^2)\right)^{1/2}.
\end{equation}
\end{theorem}

The proof of Theorem \ref{main1} follows a similar line of reasoning as in the proof of Theorem 1.1 of Chatterjee and Dembo \cite{CD1}, however the argument is more involved due to the following reasons. First, instead of having a one-sided constraint $f \geq tn$ as in Theorem 1.1, we have a two-sided constraint $|h|\leq tn$, and this calls for a minor modification of the function $\psi$. Then, more importantly, in Theorem 1.1, the upper and lower bounds are established for a probability measure, whereas here we are trying to establish the upper and lower bounds for the normalization constant of a probability measure with exponential weights. So to justify the upper bound, rather than checking the smoothness condition and the low complexity gradient condition for a single function $g$, which is connected to the constraint on $f$ as in the proof of Theorem 1.1, we need to check the smoothness condition and the low complexity gradient condition for the sum of two functions $f+e$ in our proof, where $f$ is the weight in the exponent and $e$ is connected to the constraint on $h$; while to justify the lower bound, rather than considering two small probability sets $\mathcal{A}$ and $\mathcal{A}'$ as in the proof of Theorem 1.1, we need to consider the probability of one more set $A_3$, which deals with the weight deviation in the exponent in our proof.

\vskip.2truein

\noindent \textit{Proof of the upper bound.} Let $g: \mathbb{R} \rightarrow \mathbb{R}$ be a function that is twice continuously differentiable, non-decreasing, and satisfies $g(x)=-1$ if $x \leq -1$ and $g(x)=0$ if $x \geq 0$. Let $L_1=\Vert g' \Vert$ and $L_2=\Vert g'' \Vert$. Chatterjee and Dembo \cite{CD1} described one such $g$:
\begin{equation}
g(x)=10(x+1)^3-15(x+1)^4+6(x+1)^5-1,
\end{equation}
which gives $L_1 \leq 2$ and $L_2 \leq 6$. Define
\begin{equation}
\psi(x)=Kg((t-|x|)/\delta).
\end{equation}
Then clearly $\psi(x)=-K$ if $|x|\geq t+\delta$, $\psi(x)=0$ if $|x|\leq t$, and $\psi(x)$ is non-decreasing for $-(t+\delta)\leq x\leq -t$ and non-increasing for $t\leq x\leq t+\delta$. We also have
\begin{equation}
\Vert \psi \Vert \leq K, \hspace{0.5cm} \Vert \psi' \Vert \leq \frac{2 K}{\delta}, \hspace{0.5cm} \Vert \psi'' \Vert \leq \frac{6 K}{\delta^2}.
\end{equation}

Let $e(x)=n\psi(h(x)/n)$. The plan is to apply Theorem \ref{CD1} to the function $f+e$ instead of $f$ only. Note that
\begin{equation}
\sum_{x\in \{0,1\}^n: |h(x)|\leq tn}e^{f(x)} \leq \sum_{x\in \{0,1\}^n} e^{f(x)+e(x)}.
\end{equation}
We estimate $f(x)+e(x)-I(x)$ over $[0,1]^n$. There are three cases.
\begin{itemize}
\item If $|h(x)|\leq tn$, then
\begin{equation}
f(x)+e(x)-I(x)=f(x)-I(x) \leq \sup_{x\in [0,1]^n: |h(x)|\leq (t+\delta)n}(f(x)-I(x)).
\end{equation}
\item If $|h(x)|\geq (t+\delta)n$, then
\begin{align}
f(x)+e(x)-I(x)&=f(x)-nK-I(x)\leq a+n\log 2-nK
\\
&\leq -a \leq \sup_{x\in [0,1]^n: |h(x)|\leq (t+\delta)n}(f(x)-I(x)).
\nonumber
\end{align}
\item If $|h(x)|=(t+\delta')n$ for some $0<\delta'<\delta$, then
\begin{equation}
f(x)+e(x)-I(x)\leq f(x)-I(x) \leq \sup_{x\in [0,1]^n: |h(x)|\leq (t+\delta)n}(f(x)-I(x)).
\end{equation}
\end{itemize}
This shows that
\begin{equation}
\sup_{x\in [0,1]^n}(f(x)+e(x)-I(x)) \leq \sup_{x\in [0,1]^n: |h(x)|\leq (t+\delta)n}(f(x)-I(x)).
\end{equation}

We check the smoothness condition for $f+e$ first. Note that
\begin{equation}
\Vert f+e \Vert \leq a+nK=l,
\end{equation}
and for any $i$,
\begin{equation}
\left \Vert \frac{\partial (f+e)}{\partial x_i} \right \Vert \leq b_i+\frac{2K \beta_i}{\delta}=m_i,
\end{equation}
and for any $i$, $j$,
\begin{equation}
\left \Vert \frac{\partial^2 (f+e)}{\partial x_i \partial x_j} \right \Vert \leq c_{ij}+\frac{2K \gamma_{ij}}{\delta}+\frac{6K \beta_i \beta_j}{n\delta^2}=n_{ij}.
\end{equation}

Next we check the low complexity gradient condition for $f+e$. Let
\begin{equation}
\epsilon'=\frac{\epsilon}{3\Vert \psi' \Vert} \text{ and } \tau=\frac{\epsilon}{3\left(\frac{1}{n} \sum_{i=1}^n \beta_i^2\right)^{1/2}}.
\end{equation}
Define
\begin{multline}
\D(\epsilon)=\{d^f+\theta d^h: d^f\in \D_f(\epsilon/3), d^h\in \D_h(\epsilon'), \\ \text{ and } \theta=j\tau \text{ for some integer } -\Vert \psi' \Vert /\tau<j<\Vert \psi' \Vert /\tau\}.
\end{multline}
Note that
\begin{equation}
|\D(\epsilon)| \leq \frac{2\Vert\psi'\Vert}{\tau}|\D_f(\epsilon/3)||\D_h(\epsilon')|.
\end{equation}
Let $e_i=\partial e/\partial x_i$. Take any $x\in [0,1]^n$ and choose $d^f \in \D_f(\epsilon/3)$ and $d^h\in \D_h(\epsilon')$. Choose an integer $j$ between $-\Vert \psi' \Vert /\tau$ and $\Vert \psi' \Vert /\tau$ such that $|\psi'(h(x)/n)-j\tau|\leq \tau$. Let $d=d^f+j\tau d^h$ so that $d\in \D(\epsilon)$. Then
\begin{equation}
\sum_{i=1}^n (f_i(x)+e_i(x)-d_i)^2=\sum_{i=1}^n\left((f_i(x)-d^f_i)+(\psi'(h(x)/n)h_i(x)-j\tau d^h_i)\right)^2
\end{equation}
\begin{eqnarray*}
&\leq& 3\sum_{i=1}^n (f_i(x)-d^f_i)^2+3(\psi'(h(x)/n)-j\tau)^2\sum_{i=1}^n h_i(x)^2+3\Vert \psi' \Vert^2\sum_{i=1}^n (h_i(x)-d^h_i)^2\\
&\leq& \frac{1}{3}n\epsilon^2+3\tau^2 \sum_{i=1}^n \beta_i^2+3\Vert \psi' \Vert^2 n\epsilon'^2=n\epsilon^2.
\end{eqnarray*}
Thus $\D(\epsilon)$ is a finite subset of $\mathbb{R}^n$ associated with the gradient vector of $f+e$. The proof is completed by applying Theorem \ref{CD1}. \qed

\vskip.2truein

\noindent \textit{Proof of the lower bound.} Fix any $y \in [0,1]^n$ such that
$|h(y)|\leq (t-\delta_0)n$.
Let $Y=(Y_1,\dots,Y_n)$ be a random vector with independent components, where each $Y_i$ is a $\text{Bernoulli}(y_i)$ random variable. Let $Y^{(i)}$ be the random vector $(Y_1,\dots,Y_{i-1}, 0, Y_{i+1}, \dots, Y_n)$.\ Let
\begin{equation}
A_1=\{x\in \{0,1\}^n: |h(x)| \leq tn\},
\end{equation}
\begin{equation}
A_2=\{x\in \{0,1\}^n: |g(x,y)-I(y)|\leq \epsilon_0 n\},
\end{equation}
\begin{equation}
A_3=\{x\in \{0,1\}^n: |f(x)-f(y)|\leq \eta_0 n\}.
\end{equation}
Let $A=A_1\cap A_2\cap A_3$. Then
\begin{eqnarray}
\label{lower}
\sum_{x\in \{0,1\}^n: |h(x)|\leq tn}e^{f(x)}
&=&\sum_{x\in A_1}e^{f(x)-g(x,y)+g(x,y)}\\\notag
&\geq&\sum_{x\in A}e^{f(x)-g(x,y)+g(x,y)}\\\notag
&\geq&e^{f(y)-I(y)-(\epsilon_0+\eta_0) n}\PR(Y\in A).
\end{eqnarray}

We first consider $\PR(Y\in A_1)$. Let $U=h(Y)-h(y)$. For $t\in [0,1]$ and $x\in [0,1]^n$ define $u_i(t,x)=h_i(tx+(1-t)y)$. Note that
\begin{equation}
U=\int_0^1 \sum_{i=1}^n (Y_i-y_i)u_i(t,Y)dt,
\end{equation}
which implies
\begin{equation}
\ER(U^2)=\int_0^1 \sum_{i=1}^n \ER((Y_i-y_i)u_i(t,Y)U)dt.
\end{equation}
Let $U_i=h(Y^{(i)})-h(y)$ so that $Y^{(i)}$ and $U_i$ are functions of random variables $(Y_j)_{j\neq i}$ only. By the independence of $Y_i$ and $(Y^{(i)}, U_i)$, we have
\begin{equation}
\ER((Y_i-y_i)u_i(t,Y^{(i)})U_i)=0.
\end{equation}
Therefore
\begin{eqnarray}
|\ER((Y_i-y_i)u_i(t,Y)U)|&\leq&
\ER|((u_i(t,Y)-u_i(t,Y^{(i)}))U|+\ER|u_i(t,Y^{(i)})(U-U_i)|\\\notag
&\leq&\left\Vert\frac{\partial u_i}{\partial x_i}\right\Vert\Vert U \Vert+\Vert u_i \Vert \left\Vert U-U_i \right \Vert \\\notag
&\leq& 2\alpha t\gamma_{ii}+\beta_i^2.
\end{eqnarray}
This gives
\begin{equation}
\label{YA1}
\PR(Y \in A_1^c) \leq \PR(|U|\geq \delta_0 n)\leq \frac{\ER(U^2)}{\delta_0^2 n^2} \leq \frac{\sum_{i=1}^n (\alpha \gamma_{ii}+\beta_i^2)}{\delta_0^2 n^2}=\frac{1}{6}.
\end{equation}

Next we consider $\PR(Y\in A_2)$. Note that
\begin{equation}
\ER(g(Y,y))=I(y)
\end{equation}
and \begin{eqnarray}
\VR(g(Y,y))&=&\sum_{i=1}^n \VR(Y_i\log y_i+(1-Y_i)\log(1-y_i))\\\notag
&=&\sum_{i=1}^n y_i(1-y_i)\left(\log \frac{y_i}{1-y_i}\right)^2.
\end{eqnarray}
For $x\in [0,1]$, since $|\sqrt{x}\log x|\leq 1$, we have
\begin{equation}
x(1-x)\left(\log \frac{x}{1-x}\right)^2 \leq \left(|\sqrt{x}\log x|+|\sqrt{1-x}\log(1-x)|\right)^2\leq 4.
\end{equation}
Therefore
\begin{equation}
\label{YA2}
\PR(Y\in A_2^c)\leq \PR(|g(Y,y)-I(y)|\geq \epsilon_0 n)\leq \frac{\VR (g(Y,y))}{\epsilon_0^2 n^2}\leq \frac{4}{\epsilon_0^2 n}=\frac{1}{6}.
\end{equation}

Finally we consider $\PR(Y\in A_3)$. Let $V=f(Y)-f(y)$. For $t\in [0,1]$ and $x\in [0,1]^n$ define $v_i(t,x)=f_i(tx+(1-t)y)$. Note that
\begin{equation}
V=\int_0^1 \sum_{i=1}^n (Y_i-y_i)v_i(t,Y)dt,
\end{equation}
which implies
\begin{equation}
\ER(V^2)=\int_0^1 \sum_{i=1}^n \ER((Y_i-y_i)v_i(t,Y)V)dt.
\end{equation}
Let $V_i=f(Y^{(i)})-f(y)$ so that $Y^{(i)}$ and $V_i$ are functions of random variables $(Y_j)_{j\neq i}$ only. By the independence of $Y_i$ and $(Y^{(i)}, V_i)$, we have
\begin{equation}
\ER((Y_i-y_i)v_i(t,Y^{(i)})V_i)=0.
\end{equation}
Therefore
\begin{eqnarray}
|\ER((Y_i-y_i)v_i(t,Y)V)|&\leq&
\ER|((v_i(t,Y)-v_i(t,Y^{(i)}))V|+\ER|v_i(t,Y^{(i)})(V-V_i)|\\\notag
&\leq&\left\Vert\frac{\partial v_i}{\partial x_i}\right\Vert\Vert V \Vert+\Vert v_i \Vert \left\Vert V-V_i \right \Vert \\\notag
&\leq& 2atc_{ii}+b_i^2.
\end{eqnarray}
This gives
\begin{equation}
\label{YA3}
\PR(Y \in A_3^c)\leq\PR(|V|\geq \eta_0 n) \leq \frac{\ER(V^2)}{\eta_0^2 n^2} \leq \frac{\sum_{i=1}^n (ac_{ii}+b_i^2)}{\eta_0^2 n^2}=\frac{1}{6}.
\end{equation}

Combining (\ref{YA1}), (\ref{YA2}) and (\ref{YA3}), we have
\begin{equation}
\PR(Y\in A)\geq 1-\PR(Y\in A_1^c)-\PR(Y\in A_2^c)-\PR(Y\in A_3^c)\geq \frac{1}{2}.
\end{equation}
Plugging this into (\ref{lower}) and taking supremum over $y$ completes the proof. \qed

\section{Application to exponential random graphs}
\label{application}
As mentioned earlier, we would like to apply Theorem \ref{main1} to derive the exact asymptotics for the conditional normalization constant of constrained exponential random graphs. Recall the definition of an $s$-parameter family of conditional exponential random graphs introduced earlier, where we assume that the ``ideal'' edge density of the graph is $e$. Let
\begin{equation}
f(x)=\zeta_1T_1(x)+\cdots+\zeta_sT_s(x) \text{ and } h(x)=T_1(x)-N^2e,
\end{equation}
where $T_i(x)/N^2$ is the equivalent notion of homomorphism density as defined in (\ref{T}). Let $n=\binom N2$. We compare the conditional normalization constant $\psi^{e, \zeta}_{N,t}$ (\ref{cpsi1}) for constrained exponential random graphs with the generic conditional normalization constant $F^c$ (\ref{F}). Note that the constraint $|e(G_N)-e|\leq t$ may be translated into $|T_1(x)-N^2e|\leq N^2t$, and if we further redefine $t$ to be $(1-1/N)t'/2$ then we arrive at the generic constraint $|h(x)|\leq t'n$ as in (\ref{F}). Thus $\psi^{e, \zeta}_{N,t}=F^c/N^2$. In the following we give a concrete error bound for $\psi^{e, \zeta}_{N,t}$ using the estimates in Theorem \ref{main1}. Our proof is analogous to the proof of Theorem 1.6 in Chatterjee and Dembo \cite{CD1}, where they analyzed various error bounds for the generic normalization constant obtained in Theorem 1.5 (referenced as Theorem \ref{CD1} in this paper) and applied it in the exponential setting. Instead, we analyze the various error bounds for the generic conditional normalization constant obtained in Theorem \ref{main1} and apply it in the constrained exponential setting. The rationales behind the two arguments are essentially the same, except that the argument to be presented in the proof of Theorem \ref{general} is more involved due to the imposed constraint.

\begin{theorem}
\label{general}
Let $s$ be a positive integer and $H_1,\dots,H_s$ be fixed finite simple graphs. Let $N$ be another positive integer and let $n=\binom N2$. Define $T_1,\dots,T_s$ accordingly as in the paragraph before Theorem \ref{previous}. Let $\zeta_1,\dots,\zeta_s$ be $s$ real parameters and define $\psi^{e, \zeta}_{N,t}$ as in (\ref{cpsi1}). Let $f(x)=\zeta_1 T_1(x)+\cdots+\zeta_s T_s(x)$, $B=1+|\zeta_1|+\cdots+|\zeta_s|$, and $I$ be defined as in (\ref{I}). Take $\kappa>8$. Then
\begin{equation}
\sup_{x\in [0,1]^n: |h(x)|\leq (t'-cn^{-1/(2\kappa)})n}\frac{f(x)-I(x)}{N^2}-CBN^{-1/2} \leq \psi^{e, \zeta}_{N,t}
\end{equation}
\begin{equation*}
\leq \sup_{x\in [0,1]^n: |h(x)|\leq (t'+cn^{-1/(2\kappa)})n}\frac{f(x)-I(x)}{N^2}+CB^{8/5}N^{(8-\kappa)/(5\kappa)}(\log N)^{1/5}\left(1+\frac{\log B}{\log N}\right)
\end{equation*}
\begin{equation*}
+CB^2N^{(2-\kappa)/(2\kappa)},
\end{equation*}
where $t'=2Nt/(N-1)$ and $c$ and $C$ are constants that may depend only on $H_1,\dots,H_s$ and $e$.
\end{theorem}

\begin{proof}
Chatterjee and Dembo \cite{CD1} checked that $T_i(x)$ satisfies both the smoothness condition and the low complexity gradient condition stated at the beginning of this paper, which readily implies that $f$ and $h$ satisfy the assumptions of Theorem \ref{main1}. Recall that the indexing set for quantities like $b_i$ and $\gamma_{ij}$, instead of being $\{1,\dots,n\}$, is now $\{(i,j): 1\leq i<j\leq N\}$, and for simplicity we write $(ij)$ instead of $(i, j)$. Let $a$, $b_{(ij)}$, $c_{(ij)(i'j')}$ be the supremum norms of $f$ and let $\alpha$, $\beta_{(ij)}$, $\gamma_{(ij)(i'j')}$ be the corresponding supremum norms of $h$. For any $\epsilon>0$, let $\D_f(\epsilon)$ and $\D_h(\epsilon)$ be finite subsets of $\mathbb{R}^n$ associated with the gradient vectors of $f$ and $h$ respectively.

Based on the bounds for $T_i$ (\ref{T1}) (\ref{T2}) (\ref{T3}), we derive the bounds for $f$ and $h$.
\begin{equation}
a\leq CBN^2, \hspace{0.5cm} b_{(ij)}\leq CB,
\end{equation}
\begin{equation}
c_{(ij)(i'j')}\leq \left\{%
\begin{array}{ll}
    CBN^{-1}, & \hbox{if $|\{i,j,i',j'\}|=2$ or $3$;} \\
    CBN^{-2}, & \hbox{if $|\{i,j,i',j'\}|=4$,} \\
\end{array}%
\right.
\end{equation}
\begin{equation}
|\D_f(\epsilon)|\leq \prod_{i=1}^s |\D_i(\epsilon/(\zeta_i s))|\leq \exp\left(\frac{CB^4N}{\epsilon^4}\log\frac{CB}{\epsilon}\right).
\end{equation}
\begin{equation}
\alpha\leq CN^2, \hspace{0.5cm} \beta_{(ij)}\leq C,
\end{equation}
\begin{equation}
\gamma_{(ij)(i'j')}\leq \left\{%
\begin{array}{ll}
    CN^{-1}, & \hbox{if $|\{i,j,i',j'\}|=2$ or $3$;} \\
    CN^{-2}, & \hbox{if $|\{i,j,i',j'\}|=4$,} \\
\end{array}%
\right.
\end{equation}
\begin{equation}
|\D_h(\epsilon)|=|\D_1(\epsilon)|\leq \exp\left(\frac{CN}{\epsilon^4}\log\frac{C}{\epsilon}\right).
\end{equation}

We then estimate the lower and upper error bounds for $\psi^{e, \zeta}_{N,t}$ using the bounds on $f$ and $h$ obtained above. First the lower bound:
\begin{equation}
\sum_{(ij)}ac_{(ij)(ij)}\leq CB^2N^3, \hspace{0.5cm} \sum_{(ij)}b_{(ij)}^2 \leq CB^2N^2.
\end{equation}
\begin{equation}
\sum_{(ij)}\alpha \gamma_{(ij)(ij)}\leq CN^3, \hspace{0.5cm} \sum_{(ij)}\beta_{(ij)}^2 \leq CN^2.
\end{equation}
Therefore
\begin{equation}
\label{delta0}
\delta_0 \leq cn^{-1/4}\leq cn^{-1/(2\kappa)},
\end{equation}
\begin{equation}
\frac{\epsilon_0 n+\eta_0n+\log 2}{N^2}\leq CN^{-1}+CBN^{-1/2}+CN^{-2}\leq CBN^{-1/2}.
\end{equation}
This gives
\begin{equation}
\psi^{e, \zeta}_{N,t} \geq \sup_{x\in [0,1]^n: |h(x)|\leq (t'-cn^{-1/(2\kappa)})n}\frac{f(x)-I(x)}{N^2}-CBN^{-1/2},
\end{equation}

Next the more involved upper bound: Assume that $n^{-1/4}\leq \delta\leq 1$ and $0<\epsilon\leq 1$. Since $K\leq CB$, this implies that
\begin{equation}
l\leq CBN^2, \hspace{0.5cm} m_{(ij)}\leq CB\delta^{-1},
\end{equation}
\begin{equation}
n_{(ij)(i'j')}\leq \left\{%
\begin{array}{ll}
    CBN^{-1}\delta^{-1}, & \hbox{if $|\{i,j,i',j'\}|=2$ or $3$;} \\
    CBN^{-2}\delta^{-2}, & \hbox{if $|\{i,j,i',j'\}|=4$.} \\
\end{array}%
\right.
\end{equation}
The following estimates are direct consequences of the bounds on $l$, $m_{(ij)}$, and $n_{(ij)(i'j')}$.
\begin{equation}
\sum_{(ij)}ln_{(ij)(ij)}\leq CB^2N^3\delta^{-1}, \hspace{0.5cm} \sum_{(ij)}m_{(ij)}^2\leq CB^2N^2\delta^{-2},
\end{equation}
\begin{equation}
\sum_{(ij)(i'j')}ln_{(ij)(i'j')}^2 \leq CB^3N^3\delta^{-2},
\end{equation}
\begin{equation}
\sum_{(ij)(i'j')}m_{(ij)}(m_{i'j'}+4)n_{(ij)(i'j')}\leq CB^3N^2\delta^{-4},
\end{equation}
\begin{equation}
\sum_{(ij)}n_{(ij)(ij)}^2\leq CB^2\delta^{-2}, \hspace{0.5cm} \sum_{(ij)}n_{(ij)(ij)}\leq CBN\delta^{-1}.
\end{equation}
Therefore
\begin{align}
\text{complexity term }&\leq CBN^2\delta^{-1}\epsilon+CN^2\epsilon+\log\frac{CB}{\delta\epsilon}+\frac{CB^4N}{\epsilon^4}\log\frac{CB}{\epsilon}+\frac{CB^4N}{\delta^4\epsilon^4}\log\frac{CB}{\delta\epsilon}
\\
&\leq CBN^2\delta^{-1}\epsilon+\frac{CB^4N}{\delta^4\epsilon^4}\log\frac{CB}{\delta\epsilon}.
\nonumber
\end{align}
\begin{equation}
\text{smoothness term }\leq CB^{3/2}N^{3/2}\delta^{-1}+CB^2N\delta^{-2}+CBN\delta^{-1}+C\leq CB^{2}N^{3/2}\delta^{-1}.
\end{equation}
Taking $\epsilon=(B^3\log N)/(\delta^3 N)^{1/5}$, this gives
\begin{equation}
\psi^{e, \zeta}_{N,t} \leq \sup_{x\in [0,1]^n: |h(x)|\leq (t'+\delta)n}\frac{f(x)-I(x)}{N^2}+CB^{8/5}N^{-1/5}(\log N)^{1/5}\delta^{-8/5}\left(1+\frac{\log B}{\log N}\right)
\end{equation}
\begin{equation*}
+CB^2N^{-1/2}\delta^{-1}.
\end{equation*}
For $n$ large enough, we may choose $\delta=cn^{-1/(2\kappa)}$ as in (\ref{delta0}), which yields a further simplification
\begin{equation}
\psi^{e, \zeta}_{N,t} \leq \sup_{x\in [0,1]^n: |h(x)|\leq (t'+cn^{-1/(2\kappa)})n}\frac{f(x)-I(x)}{N^2}+CB^{8/5}N^{(8-\kappa)/(5\kappa)}(\log N)^{1/5}\left(1+\frac{\log B}{\log N}\right)
\end{equation}
\begin{equation*}
+CB^2N^{(2-\kappa)/(2\kappa)}.
\end{equation*}
\end{proof}

We can do a more refined analysis of Theorem \ref{general} when $\zeta_i$'s are non-negative for $i\geq 2$.

\begin{theorem}
\label{special}
Let $s$ be a positive integer and $H_1,\dots,H_s$ be fixed finite simple graphs. Let $N$ be another positive integer and let $n=\binom N2$. Let $\zeta_1,\dots,\zeta_s$ be $s$ real parameters and suppose $\zeta_i\geq 0$ for $i\geq 2$. Define $\psi^{e, \zeta}_{N,t}$ as in (\ref{cpsi1}). Let $B=1+|\zeta_1|+\cdots+|\zeta_s|$ and $I$ be defined as in (\ref{I}). Take $\kappa>8$. Then
\begin{equation}
\label{last}
-cBN^{-1/\kappa} \leq \psi^{e, \zeta}_{N,t}-\sup_{|x-e|\leq t}\left\{\zeta_1 x+\cdots+\zeta_k x^{e(H_k)}-\frac{1}{2}I(x)\right\}
\end{equation}
\begin{equation*}
\leq CB^{8/5}N^{(8-\kappa)/(5\kappa)}(\log N)^{1/5}\left(1+\frac{\log B}{\log N}\right)
\end{equation*}
\begin{equation*}
+CB^2N^{-1/\kappa},
\end{equation*}
where $e(H_i)$ denotes the number of edges in $H_i$ and $c$ and $C$ are constants that may depend only on $H_1,\dots,H_s$, $e$, and $t$.
\end{theorem}

\begin{remark}
If $H_i$, $i\geq 2$ are all stars, then the conclusions of Theorem \ref{special} hold for any $\zeta_1,\dots,\zeta_s$.
\end{remark}

\begin{remark}
As an example, consider the case where $s=2$, $H_1$ is a single edge and $H_2$ is a triangle. Theorem \ref{special} shows that the difference between $\psi^{e, \zeta}_{N,t}$ and $\sup_{|x-e|\leq t}\left\{\zeta_1 x+\zeta_2 x^{3}-\frac{1}{2}I(x)\right\}$ tends to zero as long
as $|\zeta_1|+|\zeta_2|$ grows slower than $N^{(\kappa-8)/(8\kappa)}(\log N)^{-1/8}$, thereby allowing a small degree of sparsity for $\zeta_i$. When $\zeta_i$'s are fixed, it provides an approximation error bound of order $N^{(8-\kappa)/(5\kappa)}(\log N)^{1/5}$, substantially
better than the negative power of $\log^* N$ given by Szemer\'{e}di's lemma.
\end{remark}

\begin{proof}
Fix $t>0$. We find upper and lower bounds for
\begin{equation}
L_N=\sup_{x\in [0,1]^n: |h(x)|\leq (t'+cn^{-1/(2\kappa)})n}\frac{f(x)-I(x)}{N^2}
\end{equation}
and
\begin{equation}
M_N=\sup_{x\in [0,1]^n: |h(x)|\leq (t'-cn^{-1/(2\kappa)})n}\frac{f(x)-I(x)}{N^2}
\end{equation}
in Theorem \ref{general} when $N$ is large.

On one hand, by considering $g(x,y)=x_{ij}$ for any $(\frac{i-1}{N}, \frac{i}{N}] \times (\frac{j-1}{N}, \frac{j}{N}]$ and $i\neq j$, we have
\begin{equation}
L_N \leq \sup_{\substack{g: [0,1]^2 \rightarrow [0,1], g(x,y)=g(y,x)\\
|e(g)-e| \leq t+\frac{c}{2}n^{-1/(2\kappa)}}} \left\{\zeta_1t(H_1, g)+\cdots+\zeta_kt(H_k, g)-\frac{1}{2}\iint_{[0,1]^2}I(g(x,y))dxdy\right\},
\end{equation}
\begin{equation}
M_N \leq \sup_{\substack{g: [0,1]^2 \rightarrow [0,1], g(x,y)=g(y,x)\\
|e(g)-e| \leq t}} \left\{\zeta_1t(H_1, g)+\cdots+\zeta_kt(H_k, g)-\frac{1}{2}\iint_{[0,1]^2}I(g(x,y))dxdy\right\}.
\end{equation}
It was proved in Chatterjee and Diaconis \cite{CD} that when $\zeta_i$'s are non-negative for $i\geq 2$, the above supremum may only be attained at constant functions on $[0,1]$. Therefore
\begin{equation}
L_N \leq \sup_{
|x-e|\leq t+\frac{c}{2}n^{-1/(2\kappa)}} \left\{\zeta_1 x+\cdots+\zeta_k x^{e(H_k)}-\frac{1}{2}I(x)\right\},
\end{equation}
\begin{equation}
M_N \leq \sup_{
|x-e|\leq t} \left\{\zeta_1 x+\cdots+\zeta_k x^{e(H_k)}-\frac{1}{2}I(x)\right\}.
\end{equation}

On the other hand, by considering $g'(x,y)=x_{ij}\equiv x$ for any $i\neq j$, we have
\begin{equation}
L_N \geq \sup_{|\frac{N-1}{N}x-e|\leq t}\left\{\zeta_1 x+\cdots+\zeta_k x^{e(H_k)}-\frac{1}{2}I(x)\right\}+O(\frac{1}{N}),
\end{equation}
\begin{equation}
M_N \geq \sup_{|\frac{N-1}{N}x-e|\leq t-\frac{c}{2}n^{-1/(2\kappa)}}\left\{\zeta_1 x+\cdots+\zeta_k x^{e(H_k)}-\frac{1}{2}I(x)\right\}+O(\frac{1}{N}).
\end{equation}
The $O(1/N)$ factor comes from the following consideration. The difference between $I(g')$ and $I(x)$ is easy to estimate, while the difference between $t(H_i, g')$ and $t(H_i, x)=x^{e(H_i)}$ is caused by the zero diagonal terms $x_{ii}$. We do a broad estimate of (\ref{T}) and find that it is bounded by $c_i/N$, where $c_i$ is a constant that only depends on $H_i$. Putting everything together,
\begin{equation}
L_N=M_N=\sup_{
|x-e|\leq t} \left\{\zeta_1 x+\cdots+\zeta_k x^{e(H_k)}-\frac{1}{2}I(x)\right\}+O(\frac{1}{N^{1/\kappa}}).
\end{equation}
The rest of the proof follows.
\end{proof}

\section*{Acknowledgements}
The author thanks the anonymous referees for their helpful
comments and suggestions.

\end{document}